\documentclass[12pt,twoside]{amsart}
\usepackage{amssymb,amsmath,amsthm, amscd, enumerate, mathrsfs}
\usepackage{graphicx, hhline}
\usepackage[all]{xy}
\title[On log canonical rings]
{On the log canonical ring with Kodaira dimension two}
\author{Haidong Liu}
\date{2020/9/19, version 0.10}
\address{Peking University, Beijing International Center for Mathematical Research, 
Beijing, 100871, China}
\email{hdliu@bicmr.pku.edu.cn}

\DeclareMathOperator{\Gr}{Gr}
\DeclareMathOperator{\Supp}{Supp}
\DeclareMathOperator{\Id}{Id}

\newtheorem{thm}{Theorem}[section]
\newtheorem{lem}[thm]{Lemma}

\newtheorem{cor}[thm]{Corollary}

\theoremstyle{definition}

\newtheorem{rem}[thm]{Remark}
\newtheorem*{ack}{Acknowledgments}

\makeatletter
    
    \@addtoreset{equation}{section}
\makeatother

\begin{document}

\begin{abstract}
We prove that the log canonical ring of a projective 
log canonical pair with Kodaira dimension two 
is finitely generated. 
\end{abstract}  

\maketitle 


\section{Introduction}\label{sec1}

In this paper, we prove the following main result:

\begin{thm}[Main Theorem]\label{main-thm}
Let $(X, \Delta)$ be a projective log canonical pair such 
that $\Delta$ is a $\mathbb Q$-divisor. 
Assume that $\kappa (X, K_X+\Delta)=2$. 
Then the log canonical ring 
$$
R(X, \Delta)=\bigoplus _{m\geq 0}H^0(X, \mathcal 
O_X(\lfloor m(K_X+\Delta)\rfloor))
$$ 
is a finitely generated $\mathbb C$-algebra. 
\end{thm}

This is a special case of the well-known finite generation conjecture for log canonical pairs 
(cf. \cite[Conjecture 1.1]{fujino-liu}).
In \cite{fujino-liu}, Fujino and the author listed some recent progress on the finite generation conjecture,
and proved it under the assumptions that $(X, \Delta)$ is plt and $\kappa (X, K_X+\Delta)=2$.
Their proof used 
the {\em{lc-trivial fibration}} (in the sense of \cite{fujino-gongyo})
to deal with fibrations by using a Hodge theoretic approach
rather than by running the relative minimal model program as in \cite{bchm}.
We will follow the basic idea in \cite{fujino-liu}
and prove the main result
by using a more general kind of connectedness lemma (Subsection \ref{subsec2.2})
and the slc-trivial fibration theory (Section \ref{sec3}).

This paper can be viewed as a continuation of \cite{fujino-liu}.
Some of the notation and proofs here are the same as those in that paper,
so we recommend the interested readers to read \cite{fujino-liu} as a warm-up.

\begin{ack}
The author would like to thank Professor Osamu Fujino for answering some questions 
on the theory of slc-trivial fibrations.
He would also like to thank the referees for useful comments. 
\end{ack}

We work over $\mathbb C$, the complex number field, throughout 
this paper. We also freely use the basic 
notation of the minimal model program as in 
\cite{fujino-fundamental} and \cite{fujino-foundations}. 
A {\em{variety}} means a reduced separated scheme of finite type over $\mathbb C$.
In this paper, we do not use $\mathbb R$-divisors. 
We only use $\mathbb Q$-divisors. 

\section{Preliminaries}\label{pre}

In this section, we prepare some results needed in our proof of the main theorem.
For the notation and conventions of this paper, we refer to \cite[Preliminaries]{fujino-liu}.

\subsection{Variation of mixed Hodge structures}
Let $S$ be a path connected and locally 1-connected topological space. 
A {\em{local system}} on $S$ is a locally constant sheaf $\mathbb V$ 
of $\mathbb Q$-vector spaces on $S$ (cf. \cite[Lemma B.34]{ps}).
In particular, a {\em{constant system}} is a constant sheaf $\mathbb V$.
One example of the local systems is the so-called {\em{variation of (mixed) Hodge structures}}.
We follow the notation and definitions in \cite{fujino-fujisawa} and recommend the interested readers 
to read \cite{fujino-fujisawa} for more details.

The following theorem is taken out from \cite[Theorem 7.1]{fujino-fujisawa}.
It shows that for any  
graded polarizable variation of $\mathbb Q$-mixed Hodge structures
given in this paper,
the  $\mathcal O_S$-module $\mathcal V$
is directly defined as $\mathcal O_S\otimes \mathbb V$.
It follows that  $\alpha\colon\mathbb V \to \mathcal V:=\mathcal O_S\otimes \mathbb V$ 
is given by simply tensoring $\mathcal O_S$
and thus induces trivially an identification
$\mathcal O_S\otimes \mathbb V \simeq \mathcal V$
of $\mathcal O_S$-modules. 
We can omit the morphism $\alpha$ since there is no danger of confusion.

\begin{thm}[{\cite[Theorem 7.1]{fujino-fujisawa}}]\label{d-thm2.3}
Let $(V, T)$ be a simple normal crossing pair such that $T$ is reduced, and 
$f\colon V\to W$ a projective surjective morphism onto a smooth variety $W$. 
Assume that every stratum of $(V, T)$ dominates $W$. 
Let $\Sigma$ be a simple normal crossing divisor on $W$ such that 
every stratum of $(V, T)$ is smooth over $W^*=W\setminus \Sigma$. 
Put $V^*=f^{-1}(W^*)$, $T^*=T|_{V^*}$. 
Let $\iota\colon V^*\setminus T^*\hookrightarrow V^*$ be the natural 
open immersion. 
Then the local system $\mathbb V_k:=R^k(f|_{V^*})_*\iota_!\mathbb Q_{V^*\setminus T^*}$ 
underlies a graded polarizable admissible variation of 
$\mathbb Q$-mixed Hodge structure $\mathscr V=((\mathbb V, W), (\mathcal V, W, F), \Id)$
on $W^*$ for every $k$. 
Note that $\mathcal V_k:=\mathcal O_{W^*}\otimes\mathbb V_k$.
\end{thm}

\subsection{Connectedness lemma}\label{subsec2.2}

\cite[Section 4]{fujino-liu} showed a kind of connectedness lemma 
for plt pairs. But it is not sufficient if we try to deal with finite generation conjecture
for lc pairs. For our purposes,
we need a more general kind of connectedness lemma as follows.
Note that it is also a special case of adjunction formula for quasi-log canonical pairs 
(cf. \cite[Theorem 6.3.5]{fujino-foundations}).

\begin{lem}[Connectedness]\label{conn-lem}
Let $f\colon V\to W$ be a surjective morphism 
from a smooth projective variety $V$ onto 
a normal projective variety $W$. 
Let $B_V$ be a $\mathbb Q$-divisor 
on $V$ such that 
$K_V+B_V\sim _{\mathbb Q, f}0$, 
$(V, B_V)$ sub lc, and $\Supp B_V$ a simple 
normal crossing divisor. 
Assume that the natural map 
$$
\mathcal O_W\to f_*\mathcal O_V(\lceil -(B^{<1}_V)\rceil)
$$
is an isomorphism. 
Let $Z$ be a union of some images of stratum  of $B^{=1}_V$ 
such that $Z\subsetneq W$.
Let $S$ be the union of strata of $B^{=1}_V$ mapping into $Z$.
Assume that $S$ is a union of irreducible components of $B^{=1}_V$.
Put $K_S+B_S=(K_V+B_V)|_S$ by adjunction. 
Then $(S, B_S)$ is sub slc and the natural map 
$$
\mathcal O_Z\to g_*\mathcal O_S(\lceil -(B^{<1}_S)\rceil)
$$ 
is an isomorphism, where $g:=f|_S$. 
In particular, $S$ is connected  if $Z$ is connected. 
\end{lem}

\begin{proof}
The proof is very similar to \cite[Lemma 4.1 and Corollary 4.2]{fujino-liu}. 
We can easily check that $(S, B_S)$ is sub slc by adjunction. 
Consider the following short exact sequence 
$$
0\to \mathcal O_V(\lceil -(B^{<1}_V)\rceil -S)
\to \mathcal O_V(\lceil -(B^{<1}_V)\rceil) \to \mathcal O_S
(\lceil -B^{<1}_S\rceil)\to 0. 
$$ 
Note that $B^{<1}_V|_S=B^{<1}_S$ holds. 
By \cite[Theorem 5.6.3]{fujino-foundations} and our assumptions of $Z$ and $S$,
no lc stratum of $(V, \{B_V\}+B^{=1}_V-S)$ are mapped 
into $Z$ by $f$. 
By the same proof of \cite[Theorem 6.3.5 (i)]{fujino-foundations}, 
the natural map $\mathcal O_Z\to g_*\mathcal O_S(\lceil -B^{<1}_S\rceil)$ 
is an isomorphism. 
In particular, the natural map $\mathcal O_Z\to g_*\mathcal O_S$ 
is an isomorphism. 
This implies that $S$ is connected if $Z$ is connected.
\end{proof}

The following corollary 
allows us to remove those strata of $B^{=1}_V$ in $S$ which are not dominant onto $Z$ 
when $Z$ is normal.

\begin{cor}\label{norm-cor}
Notation as in Lemma \ref{conn-lem}. Assume further that $Z$ is irreducible and normal. 
Let $S'$ be the union of irreducible components of $B^{=1}_V$ dominant onto $Z$.
Put $K_{S'}+B_{S'}=(K_V+B_V)|_{S'}$ by adjunction. 
Then $(S', B_{S'})$ is also sub slc and the natural map 
$$
\mathcal O_Z\to g'_*\mathcal O_{S'}(\lceil -(B^{<1}_{S'})\rceil)
$$ 
is an isomorphism, where $g'\colon=f|_{S'}$. In particular, $S'$ is connected. 
\end{cor}
\begin{proof}
Consider the following commutative diagram:
$$
\xymatrix{
S' \ar[d]_{g'}\ar@{^(->}[r]^{\iota}& S\ar[d]^{g}\\ 
\widetilde{Z} \ar[r]_{p}& Z
}
$$
where $\iota \colon S'\to S $ is the natural closed immersion 
and $p$ is the normalization by \cite[Claim 1]{fujino-liu-normal}.
Since $Z$ is normal, $p$ is an isomorphism.
Then the rest of the proof is exactly the same as \cite[Claim 2]{fujino-liu-normal}.
\end{proof}

\subsection{MMP for projective dlt surfaces}
Finally, we give a special case of the minimal model 
program for projective dlt surfaces. 
It plays the same role in this paper as \cite[Lemma 6.1]{fujino-liu} in 
 \cite[Theorem 1.2]{fujino-liu}.
The proof is exactly the same as \cite[Lemma 6.1]{fujino-liu},
so we omit it here.

\begin{lem}\label{mmp-lem} 
Let $(X, B)$ be a projective 
dlt surface such that 
$B$ is a $\mathbb Q$-divisor and let $M$ be a nef $\mathbb Q$-divisor 
on $X$. 
Assume that $K_X+B+M$ is big. 
Then we can run the minimal model program with respect 
to $K_X+B+M$ and 
get a sequence of extremal contraction 
morphisms 
$$
(X, B+M)=:(X_0, B_0+M_0)\overset {\varphi_0}\longrightarrow \cdots 
\overset{\varphi_{k-1}}\longrightarrow (X_k, B_k+M_k)=:
(X^*, B^*+M^*)
$$ 
with the following properties: 
\begin{itemize}
\item[(i)] each $\varphi_i$ is a $(K_{X_i}+B_i+M_i)$-negative 
extremal birational contraction morphism, 
\item[(ii)] $K_{X_{i+1}}=\varphi_{i*}K_{X_i}$, 
$B_{i+1}=\varphi_{i*}B_i$, 
and $M_{i+1}=\varphi_{i*}M_i$ for every $i$, 
\item[(iii)] $M_i$ is nef for every $i$, and 
\item[(iv)] $K_{X^*}+B^*+M^*$ is nef 
and big. 
\end{itemize}
\end{lem}

\section{On slc-trivial fibrations}\label{sec3}
Recently, Fujino generalized the {\em{klt-trivial fibration}} in \cite{ambro-shokurov} and 
the {\em{lc-trivial fibration}} in \cite{fujino-gongyo}
to the so-called {\em{slc-trivial fibration}} in \cite{fujino-slc-trivial},
where using some deep results of theory of variations 
of mixed Hodge structures on cohomology with compact support.
For more details about slc-trivial fibrations, 
see \cite{fujino-slc-trivial} and \cite{fujino-fujisawa-liu}.

Let $f\colon V\to W$ be a projective surjective morphism from a projective simple normal crossing 
variety $V$ onto a normal projective variety $W$ such that
every stratum of $V$ is dominant onto $W$ and $f_* \mathcal O_V= \mathcal O_W$.
Let $B_V$ be a $\mathbb Q$-divisor on $V$ such that 
$(V, B_V)$ is sub slc and $\Supp B_V$ a simple normal crossing divisor. 
Put 
$$
B_W:=\sum _P (1-b_P)P, 
$$ 
where $P$ runs over prime divisors on $W$ and 
$$
b_P:=\max \left\{t \in \mathbb Q \mid
 {\text{$(V, B_V+tf^*P)$ is sub slc over the generic point of $P$}} \right\}.  
$$ 
It is easy to see that $B_W$ is a well-defined 
$\mathbb Q$-divisor on $W$ (cf. \cite[4.5]{fujino-slc-trivial}). 
We call $B_W$ the {\em{discriminant $\mathbb Q$-divisor}} 
of $f\colon (V, B_V)\to W$. 
We assume that the natural map 
$$
\mathcal O_W\to f_*\mathcal O_V(\lceil -(B^{<1}_V)\rceil)
$$ 
is an isomorphism. Then the same as \cite[Lemma 5.1]{fujino-liu}, 
we immediately get that $B_W$ is a boundary $\mathbb Q$-divisor on $W$. 

From now on, we assume that 
$K_V+B_V\sim _{\mathbb Q, f} 0$. Let $b=\min \{m\in \mathbb Z_{>0} \mid m(K_F+B_F)\sim 0\}$
where $F$ is a general fiber of $f$ and $K_F+B_F=(K_V+B_V)|_F$.
Then we can take a 
$\mathbb Q$-Cartier $\mathbb Q$-divisor $D$ on $W$ and
a rational function $\varphi\in \Gamma (V, \mathcal K^*_V)$ (see \cite[Section 6]{fujino-slc-trivial})
such that
$$
K_V+B_V+\frac{1}{b}(\varphi)=f^*D.
$$ 
Then put $$M_W:=D-K_W-B_W, $$where 
$K_W$ is the canonical divisor of $W$. 
We call $M_W$ the {\em{moduli $\mathbb Q$-divisor}} 
of
$K_V+B_V+\frac{1}{b}(\varphi)=f^*D$. 
Under above assumptions and definitions, such a morphism $f\colon (V, B_V) \to (W, D)$ 
is a kind of (basic) {\em{slc-trivial fibrations}} defined in \cite[Definition 4.1]{fujino-slc-trivial}
and we call 
$$
D=K_W+B_W+M_W
$$
the  {\em{structure decomposition}}.
Note that $D$ is uniquely determined by $\varphi$ once $K_V$, $K_W$ and $B_V$
are fixed (cf. \cite[(2.6.i)]{mori},
\cite[Proposition 4.2]{fujino-mori} or \cite[Remark 2.5]{ambro-shokurov}); thus so is $M_W$. 
Note also that $\varphi$ can be viewed as a b-divisor 
in the sense of \cite[1.2 and Example 1.1 (2)]{ambro-shokurov} or 
\cite[Definition 2.12]{fujino-slc-trivial}.

Based on the theory of slc-trivial fibrations, 
we get a useful corollary from \cite{fujino-fujisawa-liu}
which is a generalization of \cite[Theorem 0.1]{ambro-shokurov} and \cite[Theorem 1.4]{floris}.

\begin{thm}[{\cite[Corollary 1.4]{fujino-fujisawa-liu}}]\label{slc-curve}
Notation as above. If $\dim W=1$, then the moduli 
$\mathbb Q$-divisor $M_W$ is semi-ample.
\end{thm}
 
The same as \cite[Corollary 5.4]{fujino-liu}, 
we immediately get that

\begin{cor}\label{slc-cor}
Notation as above. If $\dim W=1$ and $D$ is nef, 
then $D$ is semi-ample. 
\end{cor}

The following lemma seems to be a simple fact hidden behind the proof of 
Theorem \ref{slc-curve}. But it will play a very key role in
this paper. It shows that
if the moduli part
of an slc-trivial fibration is numerically trivial, then this moduli part
 defines a local system coming from 
the variation of (mixed) Hodge structures, 
and the difference between the moduli part and the local system
is given by the rational section $\varphi$.
This property makes the moduli parts possible to be glued together
in the non-normal cases.

\begin{lem}\label{rid-lem}
Notation as above. If $\dim W=1$ and $M_W \equiv 0$, 
then there exists a  positive integer $k$ 
such that $\mathcal O_W(k M_W)\cdot (\sqrt[b]{\varphi^k}) = \mathcal O_W$.
\end{lem}
\begin{proof}
We can assume that the morphism 
$f\colon (V, B_V) \to W$ 
satisfies the following conditions (a)--(g).
They are nothing but the conditions 
stated in \cite[Proposition 6.3]{fujino-slc-trivial} and \cite[Section 5]{fujino-fujisawa-liu}: 
\begin{itemize}
\item[(a)] $W$ is a smooth curve and $V$ is a projective simple normal crossing variety. 
\item[(b)] $\Sigma_W$ and $\Sigma_V$ are simple normal crossing divisors on $W$ and $V$ respectively. 
\item[(c)] $f$ is a projective surjective morphism. 
\item[(d)]  $B_V$ and $B_W$,  $M_W$ are supported by $\Sigma_V$ and $\Sigma_W$ respectively. 
\item[(e)] every stratum of $(V, \Sigma^h_V)$ is smooth over $W^*:=W\setminus \Sigma_W$. 
\item[(f)] $f^{-1}(\Sigma_W)\subset \Sigma_V$, $f(\Sigma^v_V)\subset \Sigma_W$. 
\item[(g)] $(B^h_V)^{=1}$ is Cartier. 
\end{itemize}
By \cite[Lemma 7.3, Theorem 8.1]{fujino-slc-trivial},
there exists a finite surjective morphism $\pi\colon C \to W$ (unipotent reduction)
and a Cartier divisor $M_C$
such that $M_C=\pi^* M_W$.
Note that there is also an
induced (pre-basic) slc-trivial fibration (see \cite[4.3]{fujino-slc-trivial})
$f'\colon (V', B_{V'}) \to (C,\pi^*D)$
with 
$$
K_{V'}+B_{V'}+\frac{1}{b}(\varphi')=f'^*(\pi^*D).
$$ 
where $\varphi'$ is the pullback of $\varphi$.
By the proof of \cite[Theorem 1.3]{fujino-fujisawa-liu} 
(see also \cite[Lemma 5.2]{ambro-shokurov}),
$\mathcal O_C(M_C) \cdot (\sqrt[b]{\varphi'})|_{C^*}$ is a direct summand of 
$F^0\Gr^W_l\!\left((\mathcal V^d_{C^*})^*\right)$ 
where $C^*=\pi^{-1} (W^*)$ and $\Gr^W_l\!\left((\mathcal V^d_{C^*})^*\right)$ is 
a polarizable 
variation of $\mathbb Q$-Hodge structures. 
By definitions and Theorem \ref{d-thm2.3}, 
$$
\Gr^W_l\!\left((\mathcal V^d_{C^*})^*\right)= \mathcal O_{C^*}\otimes \mathbb V
$$
where $\mathbb V$ is a local system on $C^*$.
Note that the induced filtration $F^0( \mathbb V)$ is not necessary a local subsystem of $\mathbb V$.
But by \cite[Proposition 6.3]{fujino-slc-trivial} and the assumption that $M_C=\pi^* M_W\equiv 0$, 
there is an induced identification:
$$
\mathcal O_C(M_C) \cdot (\sqrt[b]{\varphi'})|_{C^*} =\mathcal O_{C^*}\otimes \mathbb M
$$
where $\mathbb M \subset \mathbb V$ 
is a local subsystem of rank one
by \cite[Lemma 4.8]{fujino-fujisawa-liu}.
Then by  \cite[Corollaire (4.2.8) (iii) b)]{deligne},
there is a positive integer $t$ such that  
$\mathbb M^{\otimes t}$ is a constant system
and 
$$
\mathcal O_C(tM_C) \cdot (\sqrt[b]{(\varphi')^t})|_{C^*}
=\mathcal O_{C^*}\otimes \mathbb M^{\otimes t}.
$$
Therefore, we can take a canonical extension such that 
$$
\mathcal O_C(tM_C) \cdot (\sqrt[b]{(\varphi')^t})=\mathcal O_{C}\otimes \mathbb M^{\otimes t}
$$
by \cite[Theorem 7.1]{fujino-fujisawa} (see also \cite[Lemma 1]{kawamata},
\cite[Theorem 1]{nakayama} or \cite[Theorem 2.6]{kollar}).
That is, $tM_C+t(\sqrt[b]{\varphi'})=0$ by viewing as Cartier divisors.
By pushing forward, we have that
$$
 t\cdot \deg \pi\cdot(M_W+(\sqrt[b]{\varphi}))=0.
$$
Let $k= t\cdot \deg \pi$. Then $\mathcal O_W(k M_W)\cdot (\sqrt[b]{\varphi^k}) = \mathcal O_W $.
\end{proof}

\begin{rem}\label{rem4.5}
Note that on $W$, we can show that $\mathcal O_{W}(M_W)(\sqrt[b]{\varphi})|_{W^*}$ 
defines a local subsystem
by the same proof of \cite[Theorem 1.3]{fujino-fujisawa-liu}. 
Then
by  \cite[Corollaire (4.2.8) (iii) b)]{deligne},
there is a positive integer $k$ such that  
$\mathcal O_{W}(kM_W)\cdot (\sqrt[b]{\varphi^k})|_{W^*}$ is a constant
system. This $k$ coincides with that $k$ in Lemma \ref{rid-lem}.
\end{rem}

Now we are ready to prove the following corollary.

\begin{cor}\label{key-cor} Notation as above.
If $\dim W=2$,  $(W, B_W)$ is dlt,
and $D$ is nef and big, then 
$D$ is semi-ample. 
\end{cor}

\begin{proof}
Let $C=B^{=1}_W$ and assume that $C=\sum C_i$ is connected for simplicity.
Note that $D=K_W+B_W+M_W$ where $B_W=C+B^{<1}_W$ is a boundary $\mathbb Q$-divisor and
$M_W$ is nef by \cite[Theorem 1.2]{fujino-slc-trivial}.
Then $2D-(K_W+B_W)=D+M_W$ is nef 
and big. Therefore, to prove that $D$ is semi-ample,
it suffices to prove that $D|_C$ is semi-ample
by Kawamata--Shokurov basepoint-free theorem (cf. \cite[Lemma 4.3]{fujino-liu}).
Let $C=A+B$ where $A=\sum\limits_{D\cdot C_i =0}C_i$ and $B=\sum\limits_{D\cdot C_j >0}C_j$.
Then $D|_A$ is numerically trivial on $A$ and $D|_B$ is ample on $B$. 
If we can prove that $D|_A$ is $\mathbb Q$-linearly trivial, then 
$D|_C$ is semi-ample by \cite[Lemma 2.16]{gongyo}. Note that
$D|_A$ is numerically trivial is equivalent to
$$
(K_W+A+B+B^{<1}_W+M_W)\cdot A=0.
$$
This implies that $2 p_a(A)-2=\deg K_A =(K_W+A)\cdot A\leq 0$. If $p_a(A)=0$,
then it is obvious that $D|_A$ is $\mathbb Q$-linearly trivial.
Thus we assume that $p_a(A)=1$. It follows that
$$
 (K_W+A)\cdot A=B\cdot A=B^{<1}_W\cdot A= M_W\cdot A=0.
$$
We can see that $A$ does not intersect $B$ by $ B\cdot A=0$.
Since we assume that $C$ is connected at the beginning,
It follows that $B=0$ in this case. That is, $C=A$. 
Moreover, $(K_W+C)\cdot C=(K_W+A)\cdot A=0$ implies that
 $C$ is either a smooth elliptic curve or a nodal rational curve or 
a cycle of smooth rational curves (taking analytic dlt pairs into consideration).
When $C$ is a smooth elliptic curve, 
$D|_C$ is $\mathbb Q$-linearly trivial by Corollary \ref{slc-cor}.
When $C$ is a nodal rational curve and $P$ is the nodal point,
we blow up $W$ at point $P$ and denote it as $\pi\colon W'\to W$. 
Note that $W'$ is smooth at around $\pi^{-1}(P)$ since $W$ is smooth at the nodal point $P$.
Let $B_{W'}$ be the $\mathbb Q$-divisor 
such that $K_{W'}+B_{W'}=\pi^*(K_W+B_W)$,  $D'=\pi^*D$ and $M_{W'}=\pi^*M_W$.
Then it is easy to see that $C'=B^{=1}_{W'}$ is a cycle of two smooth rational curves
and $D'$ is semi-ample if and only if $D$ is semi-ample. Thus we reduce the case to
that $C$ is a cycle of smooth rational curves.
Then $M_W\cdot A=M_W\cdot C=0$ implies that $M_W\cdot C_i=0$ for every $i$ since $M_W$ is nef.
Let $S$ be the union of strata of $B^{=1}_V$ mapping into $C$.
By further resolutions, 
we can assume that $S$ is a union of irreducible components of $B^{=1}_V$ 
(cf. \cite[Proposition 6.3.1]{fujino-foundations}).
By Lemma \ref{conn-lem}, the natural map 
$\mathcal O_{C}\to g_{*}\mathcal O_{S}(\lceil -(B^{<1}_{S})\rceil)$
is an isomorphism, 
where $K_{S}+B_{S}=(K_V+B_V)|_{S}$ and 
$g=f|_{S}$.
Similarly,
let $S_i$ be the union of irreducible components of $B^{=1}_V$ 
dominant onto (not only mapping into) $C_i$ for every $i$. 
By Lemma \ref{conn-lem} and Corollary \ref{norm-cor},
$\mathcal O_{C_i}\to {g_{i}}_*\mathcal O_{S_i}(\lceil -(B^{<1}_{S_i})\rceil)$
is an isomorphism where 
$K_{S_i}+B_{S_i}=(K_V+B_V)|_{S_i}$ and 
$g_i=f|_{S_i}$.
Then by adjunction,
$$
K_S+B_S+\frac{1}{b}(\varphi)|_S=(K_V+B_V+\frac{1}{b}(\varphi))|_S=g^*(D|_C)
$$ 
and
$$
K_{S_i}+B_{S_i}+\frac{1}{b}(\varphi)|_{S_i}=(K_V+B_V+\frac{1}{b}(\varphi))|_{S_i}=g_i^*(D|_{C_i}).
$$
Note that the number $b_i:=\min \{m\in \mathbb Z_{>0}| m(K_{F_i}+B_{F_i})\sim 0\}$
is a factor of $b$ where $F_i$ is the general fiber of $g_i$ for every $i$. 
That is, there exists a positive integer $s_i$ such that $b=s_i b_i$ for every $i$.
Then the morphism $g_i\colon (S_i, B_{S_i}) \to C_i$ 
satisfies our definition of slc-trivial fibrations with 
$K_{S_i}+B_{S_i}+\frac{1}{b_i}(\sqrt[s_i]{\varphi})|_{S_i}=g_i^*(D|_{C_i})$ and
\begin{equation}\label{eq4.1}
D|_{C_i}=K_{C_i}+B_{C_i}+M_{C_i}.
\end{equation}
By Lemma \ref{rid-lem},
there exists a positive integer $k$ (not depending on $i$)
such that  
\begin{equation*}
\mathcal O_{C_i}(kM_{C_i})\cdot  (\sqrt[b]{\varphi^k}|_{C_i}) =\mathcal O_{C_i}.
\end{equation*}
By adjunction, we have
\begin{equation}\label{eq4.2}
D|_{C_i}=(K_{W}+C+M_{W})|_{C_i}=K_{C_i}+(C-C_i)|_{C_i}+M_{W}|_{C_i}.
\end{equation}
Comparing \eqref{eq4.1} and \eqref{eq4.2}, it is easy to  get 
that $B_{C_i}=(C-C_i)|_{C_i}$ consists of two reduced points on $C_i$ 
and $M_{C_i}=M_{W}|_{C_i}$. Therefore,
\begin{equation}\label{eq4.3}
(\mathcal O_{C}(kM_W)\cdot  (\sqrt[b]{\varphi^k}|_C))|_{C_i}=
\mathcal O_{C_i}(kM_{C_i})\cdot (\sqrt[b]{\varphi^k}|_{C_i}) =\mathcal O_{C_i}.
\end{equation}
Since the right hand side is the structure sheaf for every $i$,
we can glue them together and get $\mathcal O_C$ exactly.
That is,
$\mathcal O_{C}(kM_W)\cdot  (\sqrt[b]{\varphi^k}|_C)=\mathcal O_C$.
Then $\mathcal O_{C}(M_W)\sim_{\mathbb Q} \mathcal O_C$ and thus
 $M_W|_C \sim_{\mathbb Q}0$. Therefore,
$$
D|_C=K_C+M_W|_C\sim M_W|_C\sim_{\mathbb Q} 0,
$$
and this is what we want.
\end{proof}

\begin{rem}\label{use-rem}
In fact, we showed that if $\dim W=2$,  $(W, B_W)$ is dlt, $D$ is nef
and there is some number $a>0$ such that $aD-(K_W+B_W)$ is nef and big, then 
$D$ is semi-ample. The proof is without any change.
\end{rem}

\begin{proof}[proof of the main theorem]
By the same proof of \cite[Theorem 1.2]{fujino-liu},
we can reduce to prove that the ring $R(Y, D)$ is finitely generated
for an slc-trivial fibration $f\colon (V, B_V) \to (Y,D)$ where 
$D=K_Y+B_Y+M_Y$, $(Y, B_Y)$ is dlt and $D$ is big.
By Lemma \ref{mmp-lem}, we can further assume that $D$ is nef.
Then our conclusion follows from
Corollary \ref{key-cor}.
\end{proof}


\end{document}